\documentclass[10pt]{amsart}

\pdfoutput=1
\usepackage{geometry} 
\usepackage{graphicx,url}
\usepackage{amssymb,amsmath, amsthm}
\usepackage{epstopdf}
\usepackage{subfig}
\usepackage[applemac]{inputenc} 

\newcommand{\re}{\mathbb{R}}

\newcommand{\rr}{\mathbb{R}}

\usepackage{graphicx,url}
\newtheorem{theorem}{Theorem}
\newtheorem{proposition}{Propostion}
\newtheorem{lemma}{Lemma}
\newtheorem{remark}{Remark}[section]
\newtheorem{corollary}[theorem]{Corollary}

\title{Random flights connecting Porous Medium and Euler-Poisson-Darboux equations}
\numberwithin{equation}{section}
\author{ Alessandro De Gregorio}
\author{Enzo Orsingher}

\newcommand{\de}{\mathrm{d}}

\newcommand{\xd}{{\bf x}_d}
\newcommand{\bb}{\mathbb}

\address{Dipartimento di Scienze Statistiche,
``Sapienza" University of Rome,
P.le Aldo Moro, 5 - 00185, Rome, Italy}
\email{alessandro.degregorio@uniroma1.it}
\email{enzo.orsingher@uniroma1.it}
\allowdisplaybreaks
\date{\today}
\begin{document}

\maketitle
\begin{abstract}
In this paper we consider the Porous Medium Equation and establish a relationship between its  Kompanets-Zel'dovich-Barenblatt solution $u(\xd,t), \xd\in \mathbb R^d,t>0$ and random flights. The time-rescaled version of $u(\xd,t)$ is the fundamental solution of the Euler-Poisson-Darboux equation which governs the distribution of random flights performed by a particle whose displacements have a Dirichlet probability distribution and choosing directions uniformly on a $d$-dimensional sphere (see, e.g., \cite{dgo}). 

We consider the space-fractional version of the Euler-Poisson-Darboux equation and present the solution of the related Cauchy problem in terms of the probability distributions of random flights governed by the classical Euler-Poisson-Darboux equation. Furthermore, this research is also aimed at studying the relationship between the solutions of a fractional Porous Medium Equation and the fractional  Euler-Poisson-Darboux equation. 

A considerable part of the paper is devoted to the analysis of the probabilistic tools of the solutions of the fractional equations. Also the extension to higher-order Euler-Poisson-Darboux equation is considered and the solutions interpreted as compositions of laws of pseudoprocesses.

\end{abstract}


\section{Introduction}

The starting point of our paper is the Porous Medium Equation (PME)
\begin{equation}\label{eq:pme}
\frac{\partial u}{\partial t }=\Delta (u^m),\quad m>1, d\geq 1, t>0,
\end{equation}
with $u:=u(\xd,t), \xd:=(x_1,x_2,...,x_d)\in \rr^d, t>0,\Delta:=\sum_{j=1}^d \frac{\partial^2}{\partial x_j^2}$ (Laplace operator). 
Sometimes the PME is written as
\begin{equation}\label{eq:pme2}
\frac{\partial u}{\partial t }=\nabla(mu^{m-1}\nabla u),
\end{equation}
where $\nabla:=(\frac{\partial}{\partial x_1},...,\frac{\partial}{\partial x_d})$ is the gradient operator. 

 The equation \eqref{eq:pme} is one of the simplest examples of a nonlinear evolution equation of parabolic type. It appears in the description of different phenomena 
 and its theory and properties are substantially far from those of the heat equation, $\frac{\partial u}{\partial t } = \Delta u$. Hence, this explains the interest of its study, both for the pure mathematician and for the applied scientist. There are numerous generalizations of \eqref{eq:pme} (see, e.g., \cite{vazquez}); for example in the form $\frac{\partial u}{\partial t }=\Delta f(u)$ with a suitable function $f.$ There is also the hyperbolic version of the PME in the form $\frac{\partial^2 u}{\partial t^2 }=\Delta (u^m).$

The PME was introduced with the aim of overcoming the paradox of infinite velocity of the heat flow of the classical Fourier equation. An alternative approach circumventing this paradox of propagation at infinite velocity of the heat flow is the Cattaneo-Maxwell equation (see \cite{catt}) which is substantially a telegraph equation.

The PME was also considered in the hyperbolic half-space $\mathbb H_d:=\{\xd\in\mathbb R^d:x_d>0\}$ (see \cite{vazhyp}), where $\Delta$ is replaced by the hyperbolic Laplacian
$x_d^2\Delta+(2-d)x_d,
$
also written in hyperbolic coordinates as 
$$\frac{1}{(\sinh r)^{d-1}}\frac{\partial}{\partial r}\left(\frac{\partial}{\partial r}{(\sinh r)^{d-1}\frac{\partial}{\partial r} }\right)u+\frac{1}{(\sinh r)^2}\Delta_{\mathbb S_1^{d-1}}u$$
where $\Delta_{\mathbb S_1^{d-1}}$ is the spherical Laplace-Beltrami operator. 
Space-fractional versions of the PME have been studied, for instance, in \cite{vazquez2} and \cite{biler}.

Under the initial condition $u(\xd,0)=\delta(\xd),$ the source-type solution to \eqref{eq:pme} (also called sometimes Barenblatt solution) was found by Zel’dovich, Kompanets \cite{zel}, Barenblatt \cite{bar} (and later by Pattle \cite{pattle}) and has the form
\begin{equation}\label{eq:bsolint}
u(\xd,t)=Ct^{-\alpha}\left(1-B\frac{||\xd||^2}{t^{2\beta}}\right)_+^{\frac{1}{m-1}},
\end{equation}
where $(x)_+:=\max(0,x),$ $$\alpha:=\frac{d}{2+d(m-1)},\quad\beta:=\frac\alpha d,\quad B:=\frac{\alpha(m-1)}{2md}.$$ The positive constant $C$ is a normalizing factor which will be given below. 

The aim of this paper is to present probabilistic interpretations of the solution \eqref{eq:bsolint}. In particular, we show that the PME is associated with random flights of different forms. This must be meant in the sense that the solutions \eqref{eq:bsolint} of the equation \eqref{eq:pme}, coincide, with suitable choices of the parameters, with the probability distribution of random flights for a fixed time $t>0$ and a fixed number of changes of direction. These types of stochastic processes have been studied over the years by different authors; see, e.g., \cite{stroock}, \cite{stadje}, \cite{stadje2}, \cite{dic}, \cite{OrsDG}, \cite{lecaer}, \cite{lecaer2}, \cite{DG12}, \cite{dgo}, \cite{pog}, \cite{ghosh}, \cite{garra}. They represent the stochastic motion of a particle moving in $\mathbb R^d$ with constant speed $c>0$ and direction changing at Poisson-paced epochs and uniformly oriented displacements (see \cite{stadje}, \cite{stadje2} and \cite{OrsDG}), or with Dirichlet joint distributed steps (see \cite{lecaer} and \cite{dgo}). 

The fundamental solution (which will be given by \eqref{eq:bsol}) of the PME has the same structure of the fundamental solution of the Euler-Poisson-Darboux (EPD) equation (studied, for instance, in \cite{garra2})
\begin{equation}\label{eq:epd0}
\frac{\partial^2 u}{\partial t^2}+\frac{2\gamma+d-1}{t}\frac{\partial u}{\partial t }=c^2\Delta u,\quad \gamma >0,
\end{equation}
that is
\begin{align}\label{eq:epdsol}
u(\xd,t)= \frac{\Gamma(\gamma+\frac12)}{\pi^{d/2}\Gamma(\gamma)}\frac{1}{(ct)^d}\left(1-\frac{||\xd||^2}{c^2t^2}\right)_+^{\gamma-1}.
\end{align}
By suitably changing the time scale, that is $t'=t^{\beta},$ the solution of  \eqref{eq:bsolint} takes the form \eqref{eq:epdsol}. 

The EPD equation \eqref{eq:epd0} for $d=1$ represents a generalization of the telegraph equation
\begin{equation*}
\frac{\partial^2 u}{\partial t^2}+2	\lambda(t) \frac{\partial u}{\partial t }=c^2\frac{\partial^2 u}{\partial x^2},
\end{equation*}
which is the governing equation of the probability distribution of a telegraph process where the reversals of velocity are paced by a non-homogeneous Poisson process with rate $\lambda(t)=\frac\gamma t$ (see \cite{garra2}). 
The projection of \eqref{eq:bsolint} on the one-dimensional space has instead density function which is also governed by the EPD equation \eqref{eq:epd}. The conditional probability density functions of the position of the random flights (when the number of changes of direction is fixed) are solutions of EPD equations with suitable coefficients.

In this paper we deal with also the fractional version of the EPD equation 
\begin{equation}\label{eq:fracepd}
\frac{\partial^2 u_\nu}{\partial t^2}+\frac{2\gamma+d-1}{t}\frac{\partial u_\nu}{\partial t }=-c^2(-\Delta)^{\nu/2} u_\nu,\quad \gamma >0,\nu\in(0,2],
\end{equation}
where $u_\nu:=u_{\nu}(\xi_d,t)$ and $-(-\Delta)^{\nu/2}$ is a pseudo-differential operator called fractional Laplace operator, defined as follows for $f$ belonging on the class of rapidly decreasing functions
\begin{equation*}
(-\Delta)^{\nu/2} f(\xd):=\frac{1}{(2\pi)^d}\int_{\mathbb R^d} e^{-i\langle \xi_d,\xd\rangle}||\xi_d||^\nu\hat f(\xi_d)\de \xi_d,
\end{equation*}
where $\hat f$ is the Fourier transform of $f$ and ${\bf \xi}_d\in \mathbb R^d$ (see \cite{kw} for various equivalent definitions of the fractional Laplacian).  
In \cite{ot} the space-time generalized fractional telegraph equation is analyzed and its solution interpreted as a time-changed isotropic stable process. The random time is represented as the inverse of a suitable combination of independent stable subordinators.

Let us consider
\begin{align}\label{eq:anti1}
p_1^\nu(\xd,w)&=\frac{1}{(2\pi)^d}\int_{\mathbb R^d} e^{-i\langle \xi_d,\xd\rangle}e^{i ||\xi_d||^{\nu/2}w}\hat\varphi(\xi_d)\de\xi_d
\end{align}
and
\begin{align}\label{eq:anti2}
p_2^\nu(\xd,w)&=\frac{1}{(2\pi)^d}\int_{\mathbb R^d} e^{-i\langle \xi_d,\xd\rangle}e^{-i ||\xi_d||^{\nu/2}w}\hat\varphi(\xi_d)\de\xi_d
\end{align}
where $\varphi$ represents a sufficiently regular  function.
We obtain that the solution of the Cauchy problem associated to the fractional EPD equation is given by
\begin{equation}\label{eq:solfepd}
u_\nu(\xd,t)=\int_{-ct}^{ct}g(w,t) \left(\frac{p_1^\nu(\xd,w)+p_2^\nu(\xd,w)}{2}\right)\de w
\end{equation}
where
$$g(w,t)=\frac{\Gamma(\gamma+\frac d2)}{\sqrt{\pi}\Gamma(\frac d2+\gamma-\frac12)ct}\left(1-\frac{w^2}{c^2t^2}\right)_+^{\frac d2+\gamma-\frac12-1}$$
satisfies \eqref{eq:epd0} in the one-dimensional case. The functions \eqref{eq:anti1} and \eqref{eq:anti2} are solutions of fractional Schr\"odinger-type equations treated in the papers \cite{isf} and \cite{isf2} for $\nu=2,$ where the authors constructed probabilistically based solutions. Furthermore, we show that the solution of the fractional EPD equation satisfies the following equation \begin{align*}
\frac{\partial u_\nu (\xd,t^\beta)}{\partial t}
=-(-\Delta)^{\nu/2}\int_{-\frac{t^{\beta}}{\sqrt B}}^{\frac{t^{\beta}}{\sqrt B}} ( g(w,t^\beta) )^m\left( \frac{  p_1^\nu(\xd,w)+  p_2^\nu(\xd,w)}{2}\right)\de w.
\end{align*}

For the case of the higher-order EPD equations $\frac{\partial^2 u}{\partial t^2}+\frac{2\lambda}{t}\frac{\partial u}{\partial t}=c_n\frac{\partial^{2n}u}{\partial x^{2n}}
,\lambda >0,x\in\mathbb R, c_n>0,$ we study in the last part of the paper solutions obtained as compositions of laws of pseudoprocesses (see \cite{lachal}) with the source-type solutions of \eqref{eq:epd}.

\section{Preliminaries on the Porous Medium equation}

The standard $d-$dimensional PME is the following non-linear heat equation
\begin{equation}\label{eq:pmebis}
\frac{\partial u}{\partial t }=\Delta (u^m),\quad m>1, d\geq 1, t>0,
\end{equation}
where $u:=u(\xd,t)$ is a non-negative scalar function defined on the space $\mathbb R^d\times(0,\infty)$ with $u(\xd,0)=\delta(\xd)$ (which means $u(\xd,t)\to \delta(\xd)$ as $t\to0$). Equation \eqref{eq:pmebis} is usually adopted to model the flow of a gas through a porous medium. The PME emerges also in the study of fluid mechanics where it models the filtration of an incompressible fluid through a porous stratum. Another important application of the PME concerns the heat radiation in plasmas, developed by Zel’dovich and collaborators in the early Fiftees. Other applications have been proposed in mathematical biology.  See \cite{vazquez} for an overall presentation of the physical and mathematical background of \eqref{eq:pmebis}.

The Kompanets-Zel'dovich-Barenblatt solution (also called source-type solution since $u(\xd,0)=\delta(\xd)$) represents a special solution of \eqref{eq:pmebis} and has the form (see Appendix)

\begin{equation}\label{eq:bsolbis}
u(\xd,t)=Ct^{-\alpha}\left(1-B\frac{||\xd||^2}{t^{2\beta}}\right)_+^{\frac{1}{m-1}},
\end{equation}
where $(x)_+:=\max(0,x),$ $$\alpha:=\frac{d}{2+d(m-1)},\quad\beta:=\frac\alpha d=\frac{1}{2+d(m-1)},\quad B:=\frac{\alpha(m-1)}{2md}=\frac{m-1}{2m(2+d(m-1))}.$$ The normalizing constant $C$ is chosen in such a way that $$\int_{\re^d}u(\xd,t)\de \xd=\text{area}(\mathbb S_1^{d-1})\int_0^{t^{\beta}/\sqrt B}Ct^{-\alpha}\rho^{d-1}\left(1-B\frac{\rho^2}{t^{2\beta}}\right)^{\frac{1}{m-1}}\de\rho=1$$ and this yields
\begin{equation}
C:=\frac{\Gamma(\frac d2+\frac{m}{m-1})B^{\frac d2}}{\Gamma(\frac{m}{m-1})\pi^{\frac d2}}=\frac{\Gamma(\frac d2+\frac{m}{m-1})(\frac{m-1}{2m(2+d(m-1))})^{\frac d2}}{\Gamma(\frac{m}{m-1})\pi^{\frac d2}},
\end{equation}
where area$(\mathbb{S}_1^{d-1})=\frac{2\pi^{d/2}}{\Gamma(d/2)}.$ We observe that the solution considerably simplifies for $m=2$. For $m\to 1,$ \eqref{eq:bsolbis} becomes the Gaussian kernel $(4\pi t)^{-d/2}\exp\{-x^2/4t\}$ representing the fundamental solution to the heat equation $\frac{\partial u}{\partial t}=\Delta u.$

The PME has the property of finite speed of propagation of disturbances from the rest level $u = 0.$ We are able to explain this property as follows. If we take as initial data a density distribution given by a nonnegative, bounded and compactly supported function, the physical solution of the PME for these data is a continuous function $u(\xd,t)$ such that for any $t > 0$ the profile $u( \cdot, t )$ is still nonnegative, bounded and compactly supported. Hence, the support expands eventually to penetrate the whole space, but is bounded at any fixed time. Therefore, for fixed $t>0,$ the support of \eqref{eq:bsolbis} is given by the closed ball $$D=\{\xd\in\mathbb R^d:||\xd||^2\leq t^{2\beta}/B\},$$ while the free boundary (that is the set separating the region where the solution is positive) is given by the sphere $\mathbb S^{d-1}_{t^{\beta}/\sqrt B}=\{\xd\in\mathbb R^d:||\xd||^2= t^{2\beta}/B\}.$ This implies that the Kompanets-Zel'dovich-Barenblatt solution spreads in space as $t^\beta;$ that is the radius of its spherical support increases as $t^\beta.$

The finite speed of propagation for the PME is in contrast with the infinite speed of propagation of the classical heat equation; that is a nonnegative solution of the heat equation is positive everywhere in $\mathbb R^d$.

The solution \eqref{eq:bsolbis} to the PME equation satisfies the following autosimilarity relationship
\begin{align*}\label{eq:aut}
u(\xd,t)=A^{\alpha}u(A^\beta||\xd||,At).
\end{align*}
for all real positive numbers $A.$

\begin{remark}
For $d=1,$ also the wave equation has a non-linear counterpart of the form
 \begin{equation*}
\frac{\partial^2 u}{\partial t^2 }=\frac{\partial^2}{\partial x^2} (u^m),\quad m>1, t>0,
\end{equation*}
which admits solutions 
$$u(x,t)=\left(\frac{|x|+B}{\sqrt mt+D}\right)^{\frac{2}{m-1}},\quad  x\in\mathbb R\setminus\{0\}$$
where $B$ and $D$ are arbitrary constants. This statement can be checked by easy calculations.
\end{remark}

\section{Stochastic processes related to the PME}

In this section we discuss the connection between the solutions of PME equation and random flights representing random motions with finite velocity. Therefore, we start by introducing these class of stochastic processes.

A random flight is a random motion in $\mathbb R^d$ described by a particle starting at the origin with a randomly chosen direction and with speed $c>0.$ The direction of the particle changes at each collision with some scattered obstacles where a new orientation of motion is taken. For $d\geq 2,$ all the directions are independent and have the same probability distribution. The directions are chosen uniformly on the unit-radius sphere $\mathbb S_1^{d-1}:=\{\xd\in\mathbb R^d:||\xd||=1\};$ that is, for $d\geq 3$ they possess density
\begin{equation}\label{eq:unidist}
p(\theta_1,...,\theta_{d-2},\phi):=\frac{\sin^{d-2}\theta_1\sin^{d-3}\theta_2\cdots \sin\theta_{d-2}\cos\phi}{\text{area}(\mathbb{S}_1^{d-1})},
\end{equation}
where $0\leq \phi\leq 2\pi,0\leq \theta_j\leq \pi, j=1,...,d-2.$ For $d=2,$ \eqref{eq:unidist} reduces to $p(\phi)=\frac{1}{2\pi},$ while for $d=1$ we have two possible directions alternatively taken by the moving particle.

Let $(T_k,k\in\bb N_0)$ be the sequence of the instants where the random flight changes direction with $T_0:=0.$ For the case where the times $T_k$ are governed by a homogeneous Poisson process $(N(t))_{t\geq 0},$ we have that the intertimes $\tau_{k+1}:=T_{k+1}-T_k,$ with $\tau_0:=0$ and $\tau_{n+1}:=t-\sum_{j=1}^n\tau_j,$ have joint conditional distribution given by (see, e.g., \cite{OrsDG})
$$f_1(t_1,...,t_n)=\frac{P(\tau_1\in \de t_1,...,\tau_n\in\de t_n|N(t)=n)}{\de t_1\cdots \de t_n}=\frac{n!}{t^n}\ 1_{S_n}(t_1,...,t_n),$$
where $$S_n:=\left\{(t_1,...,t_n): 0<t_j<t-\sum_{k=0}^{j-1}t_k, 1\leq j\leq n, t_0=0, t_{n+1}=t-\sum_{j=1}^nt_j\right\}.$$ For the case where the random vector $(\tau_1,...,\tau_n)$ has conditional density function
\begin{equation}\label{eq:jointdis2}
f_2(t_1,...,t_n)=\frac{\Gamma((n+1)(d-1))}{(\Gamma(d-1))^{n+1}}\frac{\prod_{j=1}^{n+1}t_j^{d-2}}{t^{(n+1)(d-1)-1}}1_{S_n}(t_1,...,t_n),\quad \text{for}\quad d\geq 2,
\end{equation}
or
\begin{equation}\label{eq:jointdis2bis}
f_3(t_1,...,t_n)=\frac{\Gamma((n+1)(\frac d2-1))}{(\Gamma(\frac d2-1))^{n+1}}\frac{\prod_{j=1}^{n+1}t_j^{\frac d2-2}}{t^{(n+1)(\frac d2-1)-1}}1_{S_n}(t_1,...,t_n), \quad \text{for}\quad  d\geq 3,
\end{equation}
we have different random flights where for all values of $n$ and for all Euclidean spaces $\mathbb R^d$ with $d\geq 2$ or $d\geq 3,$ we know (see \cite{lecaer} and \cite{dgo}) the explicit form of the conditional probability distribution of the position reached by the particle.
The distributions \eqref{eq:jointdis2} and  \eqref{eq:jointdis2bis} are rescaled Dirichlet distributions, with parameters $(d-1,...,d-1),\, d\geq 2,$ and $(\frac d2-1,...,\frac d2-1),\, d\geq 3,$ respectively.

By assuming that in the interval $[0,t]$ the motion has changed $n\in \mathbb N$ times its direction, we can describe mathematically the position ${\bf X}^n_d(t):=(X_1^n(t),...,X_d^n(t))$ reached by the particle at time $t>0.$  For $d=1,$ we have the classical telegraph process given by
\begin{equation}\label{eq:tel}
X_1^n(t):= c\sum_{k=0}^{ n} V_0 (-1)^k\tau_{k+1},
\end{equation}
where $V_0$ is a r.v. with distribution $P(V_0=+1)=P(V_0=-1)=\frac12.$ The intertimes $\tau_k$s have joint probability law $f_1.$

For $d> 1,$ the random vector ${\bf X}^n_d(t)$ has components
\begin{equation}\label{eq:definitionaddfun}
\begin{split}
&X_d^n(t)=c\sum_{k=0}^{n}\tau_{k+1}
\sin\theta_{1,k}\sin\theta_{2,k}\cdot\cdot\cdot\sin\theta_{d-2,k}\sin\phi_{k}\\
&X_{d-1}^n(t)=c\sum_{k=0}^{n}\tau_{k+1}
\sin\theta_{1,k}\sin\theta_{2,k}\cdot\cdot\cdot\sin\theta_{d-2,k}\cos\phi_{k}\\
&\cdot\cdot\cdot\\
&X_2^n(t)=c\sum_{k=0}^{n}\tau_{k+1}
\sin\theta_{1,k}\cos\theta_{2,k}\\
&X_1^n(t)=c\sum_{k=0}^{n}\tau_{k+1}\cos\theta_{1,k},
\end{split}
\end{equation}
where $c>0$ represents the velocity intensity of the motion, the r.v. $(\theta_{1,k},....,\theta_{d-2,k},\phi_k,), k=0,1,...,n,$ has density \eqref{eq:unidist}, while $\tau_k$s have joint density function \eqref{eq:jointdis2}, for $d\geq 2,$ or $\eqref{eq:jointdis2bis},$ for $d\geq 3.$

Fixed $t>0$ as well as the number of changes of velocity $n\in \mathbb N,$  the telegraph process \eqref{eq:tel} admits conditional density given by (see \cite{dgos})
\begin{equation}\label{eq:disttp}
\frac{P( X_1^n(t)\in \de x_1)}{\de x_1}=\begin{cases}
\frac{\Gamma(n+1)}{(\Gamma(\frac{n+1}{2}))^22^{n}ct}\left(1-\frac{x_1^2}{c^2t^2}\right)_+^\frac{n-1}{2},& n\, \text{odd},\\
\frac{\Gamma(n+1)}{\Gamma(\frac n2+1)\Gamma(\frac n2)2^{n}ct}\left(1-\frac{x_1^2}{c^2t^2}\right)_+^{\frac n2-1},&  n\, \text{even}.
\end{cases}
\end{equation}
We observe that for $n$ odd, we have that $$P( X_1^n(t)\in \de x_1)=P(X_1^{n+1}(t)\in \de x_1).$$ 
It is worth to mention that the results \eqref{eq:disttp} could be obtained by taking into account the same approach in \cite{lp} based on the Stieltjes transform.

For $t>0,$ under the assumptions \eqref{eq:jointdis2} and \eqref{eq:jointdis2bis}, \cite{lecaer} and \cite{dgo} obtained the explicit density functions of the random flights ${\bf X}_d^n(t)$ given the number of changes of direction; that is 
\begin{equation}\label{eq:distrf}
\frac{P({\bf X}_d^n(t)\in \de \xd)}{\de \xd}=\begin{cases}
\frac{\Gamma(\frac{n+1}{2}(d-1)+\frac12)}{\Gamma(\frac{n}{2}(d-1))\pi^{\frac d2}(ct)^d}\left(1-\frac{||\xd||^2}{c^2t^2}\right)_+^{\frac n2(d-1)-1},&\text{if \eqref{eq:jointdis2} holds} , \\
\frac{\Gamma((n+1)(\frac d2-1)+1)}{\Gamma(n(\frac d2-1))\pi^{\frac d2}(ct)^d}\left(1-\frac{||\xd||^2}{c^2t^2}\right)_+^{n(\frac d2-1)-1},&\text{if \eqref{eq:jointdis2bis} holds}.
\end{cases}
\end{equation}

Now, we are able to relate the PME equation with a time rescaled version of the conditioned random flight processes $({\bf X}_d^n(t), t\geq 0)$ when some constraints on the parameters hold true. More precisely, in the next theorem, we establish relationships between the degree $m$ of the PME and the number $n$ of changes of direction of the random flights described above and the dimension $d$ of the Euclidean space where the random flights develop.

\begin{theorem}\label{teo:sppme}
The random flights $({\bf Y}_d^n(t),t\geq 0),$ where ${\bf Y}_d^n(t):= {\bf X}_d^n(t^\beta),$ admit probability density function, for a fixed time $t>0$ and $n\in\mathbb N,$ given by
\begin{equation}\label{eq:rescrfpdf}
\frac{P({\bf Y}_d^n(t)\in \de\xd)}{\de \xd}=u(\xd,t) ,
\end{equation}
where $u(\xd,t)$ is the solution \eqref{eq:bsolbis}. 
The equality \eqref{eq:rescrfpdf} is true, when the following relationships between the number $n$ of velocity changes of the random motion and the parameter $m$ of the PME fulfill:
\begin{itemize}
\item[(i)] for $d=1,$ $$m=\begin{cases}
\frac{n+1}{n-1}=1+\frac1k,& n=2k+1,\\
\frac{n}{n-2}=1+\frac{1}{k},& n=2k+2,
\end{cases}\quad k\geq 1;$$
\item[(ii)] when \eqref{eq:jointdis2} holds, $m=\frac{n(d-1)}{n(d-1)-2}$ with $d>\frac2n+1;$
\item[(iii)] when \eqref{eq:jointdis2bis} holds, $m=\frac{n(d-2)}{n(d-2)-2}$ with $d>\frac2n+2$.
\end{itemize}

\end{theorem}

\begin{proof}

For $d=1,$ we deal with a telegraph process defined by \eqref{eq:definitionaddfun} with time scale $t'=t^\beta$ and speed $c'=1/\sqrt B.$ By exploiting the duplication formula for the Gamma function we can write the solution \eqref{eq:fundsol2} for $d=1$ as follows
\begin{equation}\label{eq:epd1}
u(x_1,t')=\frac{\Gamma(\frac{2m}{m-1})2^{1-\frac{2m}{m-1}}}{(\Gamma(\frac{m}{m-1}))^2}\frac{1}{c't'}\left(1-\frac{x_1^2}{(c' t')^2}\right)_+^{\frac{1}{m-1}}.
\end{equation}
For
\begin{equation*}
\frac{1}{m-1}=\frac{n-1}{2}, \quad \text{that is}\quad
 m=\frac{n+1}{n-1}, 
\end{equation*}
the solution \eqref{eq:bsolbis} coincides with the first one of  \eqref{eq:disttp}, while for 
\begin{equation*}
\frac{1}{m-1}=\frac{n}{2}-1, \quad \text{that is}\quad
 m=\frac{n}{n-2}, 
\end{equation*}
the solution \eqref{eq:bsolbis} coincides with the second one of  \eqref{eq:disttp}.
 For $n>2,$ in both cases $1<m<\infty.$

Now,
let us consider a random flight defined in $\mathbb R^d, d\geq 2,$ by \eqref{eq:definitionaddfun} with time scale $t'=t^\beta$ and speed $c'=1/\sqrt B.$ Under the assumption \eqref{eq:jointdis2}, for
\begin{equation*}
\frac{m}{m-1}=\frac n2(d-1), \quad \text{that is}\quad m=\frac{n(d-1)}{n(d-1)-2}, 
\end{equation*}
the function \eqref{eq:fundsol2} coincides with the first one of \eqref{eq:distrf}. Since $m\in(1,\infty),$ we infer that
\begin{equation}\label{eq:ineq}
d>\frac2n+1.
\end{equation}
For $d=2$ the inequality \eqref{eq:ineq} holds for $n\geq 3;$ for $d=3,$ it holds for $n\geq 2;$ for $d>3,$ \eqref{eq:ineq} holds for all $n\geq 1.$
Therefore, under the condition  \eqref{eq:ineq} we can write
\begin{equation*}
P({\bf X}_d^n(t')\in \de	\xd)= u(\xd,t') \de \xd.
\end{equation*}
Analogously, under the assumption \eqref{eq:jointdis2bis}, for
\begin{equation*}
\frac{m}{m-1}= n\left(\frac d2-1\right), \quad \text{that is}\quad m=\frac{n(d-2)}{n(d-2)-2}, 
\end{equation*}
the function \eqref{eq:fundsol2} coincides with the second one of \eqref{eq:distrf}. Since $m\in(1,\infty),$ we infer that
\begin{equation}\label{eq:ineq2}
d>\frac2n+2.
\end{equation}
For $d=3$ the inequality \eqref{eq:ineq} holds for $n\geq 3;$ for $d=4,$ it holds for $n\geq 2;$ for $d>4,$ \eqref{eq:ineq} holds for all $n\geq 1.$
Therefore, under the condition  \eqref{eq:ineq2} it turns out that
\begin{equation*}
P({\bf X}_d^n(t')\in \de	\xd)= u(\xd,t') \de \xd.
\end{equation*}
\end{proof}
\begin{remark}\label{inv}
From Theorem \ref{teo:sppme} follows that $({\bf Y}_d^n(t),t\geq 0)$ is rotationally invariant; that is if $O(d)$ is the group of $d\times d$ orthogonal matrices acting in $\mathbb R^d,$ we have that $u(M^T\xd,t)=u(\xd,t)=u(||\xd||,t),$
where $M\in O(d).$ 
\end{remark}

\begin{remark}\label{rem:sde}
The study of stochastic processes associated with the PME has been also developed in \cite{inoue}, \cite{inoue2}, \cite{inoue3} and \cite{ben}. They proved that there exists a non-linear diffusion process $\{({\bf Z}_d(t)=(Z_1(t),...,Z_d(t))),t\geq 0), P\}$ satisfying the following system of stochastic differential equations
\begin{equation}
Z_j(t)=\int_0^t u(Z_j(s),s)^{(m-1)/2} \de B_j(t),\quad j=1,2,...,d,
\end{equation}
where $(B_1(t),...,B_d(t))$ is a standard $d$-dimensional Brownian motion such that
$$P({\bf Z}_d(t)\in\de\xd)=u(\xd,t)\de\xd,$$
with $u(\xd,t)$ given by \eqref{eq:bsolbis}.
\end{remark}
The following result concerning the Fourier transform of $u(\xd,t)$ has also been proved in \cite{biler}.

\begin{theorem}\label{propcf}
The Fourier transform of the probability law $u(\xd,t)$ given by \eqref{eq:bsolbis}, denoted by $\hat u(\xi_d,t),$ is equal to
\begin{equation}\label{eq:cfbsol}
\hat u(\xi_d,t)=\left(\frac{2\sqrt B}{t^\beta||\xi_d||}\right)^{\frac{d}{2\alpha(m-1)}}\Gamma\left(\frac d2+\frac{m}{m-1}\right)J_{\frac{d}{2\alpha(m-1)}}\left(\frac{||\xi_d||t^{\beta}}{\sqrt B}\right),
\end{equation}
where $\xi_d\in\mathbb R^d, \alpha=\frac{d}{2+d(m-1)}, \beta=\frac{1}{2+d(m-1)}$ and $J_\mu(x)=\sum_{k=0}^\infty (-1)^k\frac{(x/2)^{2k+\mu}}{k!\Gamma(k+\mu+1)}$, with $\mu\in \mathbb R,$ is the Bessel function. 
\end{theorem}
\begin{proof}
Let $\sigma$ be the uniform measure on $\mathbb S_{1}^{d-1}.$ We recall that (see (2.12), pag.690, \cite{dgo}),  
\begin{equation}\label{eq:int}
 \int_{\mathbb S_1^{d-1}}e^{i\rho\langle \xi_d,\theta_d\rangle} \de \sigma({\bf \theta}_d)=(2\pi)^{d/2}\frac{J_{\frac d2-1}(\rho||\xi_d||)}{(\rho||\xi_d||)^{\frac d2-1}}
\end{equation}
One has that
\begin{align*}
\hat u(\xi_d,t)&=\int_{\mathbb R^d}e^{i\langle\xi_d,\xd\rangle}u({\bf x}_d,t)\de {\bf x}_d\\
&=(\text{by Remark \ref{inv}})\\
&=\int_0^{\frac{t^{\beta}}{\sqrt B}}\rho^{d-1} Ct^{-\alpha} \left(1-\frac{B\rho^2}{t^{2\beta}}\right)^{\frac{1}{m-1}}\de \rho \int_{\mathbb S_1^{d-1}}e^{i\rho\langle \xi_d,\theta_d\rangle} \de \sigma({\bf \theta}_d)\\
&=(\text{by}\, \eqref{eq:int})\\
&=(2\pi)^{d/2}\int_0^{\frac{t^\beta}{\sqrt B}}\rho^{d-1} Ct^{-\alpha} \left(1-\frac{B\rho^2}{t^{2\beta}}\right)^{\frac{1}{m-1}}\frac{J_{\frac d2-1}(\rho||\xi_d||)}{(\rho||\xi_d||)^{\frac d2-1}}\de \rho\\
&=\frac{(2\pi)^{d/2} Ct^{-\alpha+\beta(\frac d2+1)}}{(\sqrt B)^{\frac d2+1}||\xi_d||^{\frac d2-1}}\int_0^1(1-w^2)^{\frac{1}{m-1}} w^{d/2}J_{d/2-1}\left(\frac{||\xi_d||t^{\beta}}{\sqrt B}w\right)\de w.
\end{align*}
In view of formula 6.567(1) of page 688 of \cite{gr}
\begin{equation}\label{eq:gr}
\int_0^1x^{\nu+1}(1-x^2)^\mu J_\nu(bx)\de x=2^\mu\Gamma(\mu+1)b^{-(\mu+1)}J_{\nu+\mu+1}(b)
\end{equation}
where $b>0,$ Re$\nu>-1,$ Re$\mu>-1,$ we obtain \eqref{eq:cfbsol}.
\end{proof}

The previous discussion entails that \eqref{eq:cfbsol} is the characteristic function of $({\bf Y}_d^n(t),t\geq 0)$ under the conditions of Theorem \ref{teo:sppme}.

\begin{proposition}
For $d=1,$ the $p$-th moment of $({\bf Y}_1^n(t),t\geq 0)$ (under the condition $(i)$ of Theorem \ref{teo:sppme} relating $m$ and $n$) is equal to
\begin{equation}\label{eq:mom}
E[{\bf Y}_1^n(t)]^p=
\begin{cases}
0,& p\, \text{odd},\\
\frac{\Gamma(\frac{p+1}{2})\Gamma(\frac12+\frac{m}{m-1})}{\sqrt\pi \Gamma(\frac{p+1}{2}+\frac{m}{m-1})}\left(\frac{t^\alpha}{\sqrt B}\right)^{p},& p\, \text{even}.
\end{cases}
\end{equation}
with $\alpha=\frac{1}{1+m}$ and $B=\frac{m-1}{2m(m+1)}.$
\end{proposition}
\begin{proof}
We have to compute
\begin{align}
E[{\bf Y}_1^n(t)]^p=Ct^{-\alpha}\int_{-t^\beta/\sqrt B}^{t^\beta/\sqrt B} x_1^p\left(1-B\frac{x_1^2}{t^{2\beta}}\right)^{\frac{1}{m-1}}\de x_1
\end{align}
Simple calculations lead to \eqref{eq:mom}.
\end{proof}
From the above proposition we derive 
\begin{equation}\label{eq:var}Var[{\bf Y}_1^n(t)]=\frac{2m(m+1)}{3m-1}t^{\frac{2}{1+m}}=\frac{2n(n+1)}{(n-1)(n+2)}t^{\frac{n-1}{n}}.\end{equation}
Since $\frac{n-1}{n}<1,$ from \eqref{eq:var} we can conclude that the process $({\bf Y}_1^n(t),t\geq 0)$ spreads like a sub-diffusion.

In Theorem 3 in \cite{dgo} the moments of order $p$ of $(||{\bf Y}_d^n(t)||,t\geq 0),$ with $d\geq 2,$ (with suitable changes of the parameters) have been evaluated; then also the multidimensional random flights  $({\bf Y}_d^n(t),t\geq 0)$ represent anomalous diffusions.

\section{Fractional EPD equation and related  solutions}
We now consider the space-fractional version of the EPD equation, that is
\begin{equation}\label{fracepd}
\begin{cases}
\left(\frac{\partial^2}{\partial t^2}+\frac{d+2\gamma-1}{t}\frac{\partial}{\partial t}\right)u_\nu=-c^2(-\Delta)^{\frac\nu 2} u_\nu,\\
u(\xd,0)=\varphi(\xd), \\
\left.\frac{\partial u(\xd,t)}{\partial t}\right|_{t=0}=0,
\end{cases}
\end{equation}
where $u_\nu:=u_{\nu}(\xd,t)$ and $0<\nu<2.$  We assume that $\varphi \in \mathcal S(\mathbb R^d)$ representing the Schwartz space of rapidly decreasing functions on $\mathbb R^d.$
The fractional Laplace operator $(-\Delta)^{\nu/2}$
 is a non-local  pseudo-differential operator, 
defined via Fourier multipliers as follows 
\begin{equation}\label{eq:deffraclap}
(-\Delta)^{\nu/2} f(\xd):=\frac{1}{(2\pi)^d}\int_{\mathbb R^d} e^{-i\langle \xi_d,\xd\rangle}||\xi_d||^\nu\hat f(\xi_d)\de \xi_d.
\end{equation}
Formally, let $L^p(\mathbb R^d),p\in[1,2];$ we say that $f\in$ Dom$((-\Delta)^{\nu/2},L^p(\mathbb R^d))$ whenever $f\in L^p(\mathbb R^d)$ and there is $(-\Delta)^{\nu/2} f(\xd)\in L^p(\mathbb R^d)$ such that  \eqref{eq:deffraclap} holds (see Definition 2.1 in \cite{kw}).
For $d=1$ the operator $(-\Delta)^{\nu/2}$ coincides with the Riesz fractional derivative
$$\frac{\partial ^\nu f(x)}{\partial |x|^\nu}=-\frac{1}{2\cos (\pi \nu/2)}\frac{1}{\Gamma(m-\nu)}\frac{\de^m}{\de x^m}\int_{\mathbb R}\frac{f(y)}{|x-y|^{\nu+1-m}}\de y,\quad m-1<\nu<m, m\in\mathbb N.$$ 
An alternative definition of the fractional Laplace operator is the following one.  Let $({\bf S}_\nu(t),t\geq 0),$ with $0<\nu<2,$ be an isotropic, $d$-dimensional, $\nu$-stable process with
\begin{equation}\label{eq:chstabsub}
E\left[e^{i\langle \xi_d,{\bf S}_\nu(t)\rangle}\right]=e^{-t||\xi_d||^{\nu}}.
\end{equation}
Let $(T_t f)(\xd)=E[f({\bf S}_\nu(t) +\xd)]$ be the semigroup associated to ${\bf S}_\nu(t), t\geq 0.$ The infinitesimal generator of the semigroup $T_t f$ is given by limit in norm, if it exists,  $\mathcal L f= \lim_{t\to 0^+}\frac{T_t f-f}{t},$ where $f\in$ Dom$(\mathcal L)$. It is well-known that $\mathcal L=-(-\Delta)^{\nu/2}$ and its domain becomes Dom$(\mathcal L)=\left\{f\in L^2(\mathbb R^d):\int_{\mathbb R^d}(1+||\xi_d||^\nu)||\hat f(\xi_d)||^2\de\xi_d<\infty\right\}$ (see, e.g., \cite{app}).

 The reader can consult \cite{kw} for a discussion on the definition of the fractional Laplace operator. Moreover, if $\nu_1>0,$ $\nu_2>0$ and $0<\nu_1+\nu_2\leq 2,$  for the fractional Laplacian the following semigroup property holds
\begin{equation}\label{eq:sem}
(-\Delta)^{\frac{\nu_1}{2}}(-\Delta)^{\frac{\nu_2}{2}}=(-\Delta)^{\frac{\nu_1}{2}+\frac{\nu_2}{2}}.
\end{equation}
Indeed, by definition \eqref{eq:deffraclap}, under the previous constraints on $\nu_1$ and $\nu_2,$ it is easy to verify that
\begin{equation*}
\int_{\mathbb R^d}e^{i\langle \xi_d,\xd\rangle}(-\Delta)^\frac{\nu_1}{2}(-\Delta)^\frac{\nu_2}{2}f(\xd)\de \xd=||\xi_d||^{\frac{\nu_1}{2}+\frac{\nu_2}{2}}\hat f(\xi_d),
\end{equation*}
and then the equality \eqref{eq:sem} is proved. For this reason, hereafter we adopt the definition \eqref{eq:deffraclap}.

\begin{remark}

 For $\nu=2,$ the equation \eqref{fracepd} coincides with $d$-dimensional EPD equation which is itself a special case of the multidimensional telegraph equation
\begin{equation}\label{eq:epd}
\frac{\partial^2 u}{\partial t^2}+2\lambda(t)\frac{\partial u}{\partial t}=c^2\Delta  u,
\end{equation}
with $\lambda(t)=\frac{d+2\gamma-1}{2t}.$ The previous equation has fundamental solution (see \cite{garra2})
\begin{equation}\label{eq:solepd}
p(\xd,t)=\frac{\Gamma(\gamma+\frac d2)}{\pi^{d/2}\Gamma(\gamma)}\frac{1}{(ct)^d}\left(1-\frac{||\xd||^2}{(ct)^2}\right)_+^{\gamma-1}
\end{equation}
where $\gamma>0.$ The EPD equation has different forms and emerged in different contexts (e.g. fluid dynamics and geometry). For $d=1,$ a probabilistic derivation of the EPD equation based on the integrated telegraph process where the reversals of velocity are paced by a non-homogeneous Poisson process with rate $\lambda(t)=\frac\gamma t$ is given in \cite{glu}.
By rescaling the time coordinate as follows  $$t':=t^\beta,$$
the solution \eqref{eq:bsolbis} of the PME coincides with \eqref{eq:solepd}. Indeed, in the frame  $(\xd,t'),$ the Kompanets-Zel'dovich-Barenblatt solution \eqref{eq:bsolbis} can be written as
\begin{equation}\label{eq:fundsol2}
u(\xd,t')=p(\xd,t')=\frac{\Gamma(\frac d2+\frac{m}{m-1})}{\Gamma(\frac{m}{m-1})\pi^{\frac d2}}\frac{1}{(c't')^{d}}\left(1-\frac{||\xd||^2}{(c' t')^2}\right)_+^{\frac{1}{m-1}},
\end{equation}
where
\begin{equation*}
c':=\frac{1}{\sqrt B}=\sqrt{\frac{2m(2+d(m-1))}{m-1}},
\end{equation*}
and solves
the EPD equation \eqref{eq:epd} with $\gamma=\frac{m}{m-1};$ that is
\begin{equation}\label{epd2}
\left(\frac{\partial^2}{\partial t'^2}+\frac{d+\frac{m+1}{m-1}}{t'}\frac{\partial}{\partial t'}\right)u(\xd,t')=c'^2\Delta u(\xd,t').
\end{equation}

We observe that in the original time variable $t=(t')^{1/\beta}$ the solution \eqref{eq:fundsol2}, that is \eqref{eq:bsolbis}, satisfies the EPD equation
\begin{equation*}
\left[\frac{\partial^2}{\partial t^2}+\frac\beta t\left(d+\frac{m+1}{m-1} -\beta(\beta-1)t^{2\beta-1}\right)\frac{\partial}{\partial t}\right]u(\xd,t)=\frac{\beta^2t^{2\beta-2}}{B}\Delta u(\xd,t).
\end{equation*}
\end{remark}
\begin{remark}
The  $k$-th one-dimensional marginals of \eqref{eq:fundsol2} have the form
\begin{equation}\label{eq:fundsolmarg}
p_k(x_k,t')=\frac{\Gamma(\frac d2+\frac{m}{m-1})}{\Gamma(\frac{m}{m-1}+\frac{d-1}{2})\pi^{\frac 12}}\frac{1}{c't'}\left(1-\frac{x_k^2}{(c' t')^2}\right)_+^{\frac{m}{m-1}+\frac{d-1}{2}},
\end{equation}
and satisfy the EPD equation
\begin{equation*}
\left(\frac{\partial^2}{\partial t'^2}+\frac{d+\frac{m+1}{m-1}}{t'}\frac{\partial}{\partial t'}\right)p_k(x_k,t')=c'^2\frac{\partial ^2p_k(x_k,t')}{\partial x_k^2}
\end{equation*}
(on this point the  reader can consult \cite{garra2}).
For $t'=t^\beta$ and $c'=\frac{1}{\sqrt B}$ the probability density \eqref{eq:fundsolmarg} coincides with the Barenblatt solution of the PME. Therefore, for $d=1$ we can interpret \eqref{eq:bsolbis} as the law of a time-rescaled telegraph process on the real line observed on $(0,t],$ where the particle alternates the two directions at epochs of a non-homogeneous Poisson process with rate $\lambda(t)=\frac{d+\frac{m+1}{m-1}}{2t^\beta}.$ Unlike the classical telegraph process the process with probability distribution \eqref{eq:fundsolmarg} has only the absolutely continuous component. It is important to underline that this interpretation holds also for all $m>1.$
\end{remark}


\begin{lemma}
If we deal with $\varphi({\bf x}_d)=\delta({\bf x}_d)$, the characteristic function of $u_{\nu}$ becomes
\begin{equation}\label{eq:ffsolepd}
\hat u_{\nu}(\xi_d,t)=\left(\frac{2}{ct||\xi_d||^{\nu/2}}\right)^{\gamma+\frac d2-1}\Gamma\left(\gamma+\frac d2\right)J_{\gamma+\frac d2-1}\left(ct||\xi_d||^{\nu/2}\right).
\end{equation}
\end{lemma}
\begin{proof}
By means of the same arguments adopted in the proof of Theorem \ref{propcf}, we are able to prove that the Fourier transform of \eqref{eq:solepd} becomes
$$\hat u(\xi_d,t)=\left(\frac{2}{ct||\xi_d||}\right)^{\gamma+\frac d2-1}\Gamma\left(\gamma+\frac d2\right)J_{\gamma+\frac d2-1}\left(ct||\xi_d||\right).$$
The Fourier transform $\hat u:=\hat u(\xi_d,t)$ satisfies
\begin{equation}\label{eq:fepd}
\frac{\partial^2 \hat u}{\partial t^2}+\frac{d+2\gamma-1}{t}\frac{\partial \hat u}{\partial t}=-c^2||\xi_d||^2\hat u,
\end{equation}
while the Fourier transform  $\hat u_\nu:=\hat u_\nu(\xi_d,t)$ of the fundamental solution to \eqref{fracepd}, is solution to
\begin{equation}\label{eq:ffepd}
\frac{\partial^2 \hat u_\nu}{\partial t^2}+\frac{d+2\gamma-1}{t}\frac{\partial \hat u_\nu}{\partial t}=-c^2||\xi_d||^\nu \hat u_\nu.
\end{equation}

Since $\hat u(\xi_d,t)$ is solution to \eqref{eq:fepd}, it follows that \eqref{eq:ffsolepd} solves \eqref{eq:ffepd}.
\end{proof}

We are able to present two different forms of the inverse Fourier transform $u_\nu(\xd,t)$ of \eqref{eq:ffsolepd}.
\begin{theorem}
We have that the inverse Fourier transform of \eqref{eq:ffsolepd} becomes
\begin{align}\label{eq:fcf1}
u_\nu(\xd,t)=\left(\frac{2}{ct}\right)^\gamma\frac{\Gamma(\gamma+\frac d2)}{\nu\pi^{d/2}(ct||\xd||)^{\frac d2-1}}\int_0^\infty \rho^{d(\frac{2-\nu}{2\nu})+\frac2\nu-\gamma}J_{\gamma+\frac d2-1}\left(ct\rho\right)J_{\frac d2-1}\left(\rho^{2/\nu}||\xd||\right)\de\rho
\end{align}
or
\begin{align}\label{eq:fcf2}
u_\nu(\xd,t)&=\frac{2\Gamma(\gamma+\frac d2)}{\nu\pi^{d/2}(ct||\xd||)^{\frac d2-1}\Gamma(\gamma)}\\
&\quad\times\int_0^1 w^{\frac d2}(1-w^2)^{\gamma-1}\de w\int_0^\infty \rho^{d(\frac{2-\nu}{2\nu})+\frac2\nu}J_{\frac d2-1}\left(ct\rho w\right)J_{\frac d2-1}\left(\rho^{2/\nu}||\xd||\right)\de\rho\notag
\end{align}
\end{theorem}
\begin{proof}
We start from \eqref{eq:ffsolepd} and write (in view of formula \eqref{eq:int})
\begin{align*}
 u_\nu(\xd,t)&=\frac{1}{(2\pi)^d}\int_{\mathbb R^d}e^{-i\langle \xi_d,\xd\rangle}\hat u_\nu(\xi_d,t)\de \xi_d\\
&= \frac{1}{(2\pi)^d}\int_0^\infty\rho^{d-1} \left(\frac{2}{ct\rho^{\nu/2}}\right)^{\gamma+\frac d2-1}\Gamma\left(\gamma+\frac d2\right)J_{\gamma+\frac d2-1}\left(ct\rho^{\nu/2}\right)\de\rho\int_{\mathbb S_1^{d-1}}e^{-i\rho\langle \theta_d,\xd\rangle} \de \sigma({\bf \theta}_d)\\
&=(by\, \eqref{eq:int})\\
&=\frac{1}{(2\pi)^{d/2}}\int_0^\infty \rho^{d-1}\left(\frac{2}{ct\rho^{\nu/2}}\right)^{\gamma+\frac d2-1}\Gamma\left(\gamma+\frac d2\right)J_{\gamma+\frac d2-1}\left(ct\rho^{\nu/2}\right)\frac{J_{\frac d2-1}(\rho||\xd||)}{(\rho||\xd||)^{\frac d2-1}}\de\rho\\
&=\frac{1}{(2\pi)^{d/2}}\left(\frac{2}{ct}\right)^{\gamma+\frac d2-1}\frac{\Gamma\left(\gamma+\frac d2\right)}{||\xd||^{\frac d2-1}}\int_0^\infty\rho^{\frac d2-\frac\nu2(\frac d2-1+\gamma)}J_{\gamma+\frac d2-1}\left(ct\rho^{\nu/2}\right)J_{\frac d2-1}(\rho||\xd||)\de\rho
\end{align*}
By a simple change of variable in the above integral we obtain the result \eqref{eq:fcf1}.

In order to obtain the result \eqref{eq:fcf2}, we apply \eqref{eq:gr}
\begin{equation}\label{eq:intbess}
\int_0^1 w^{\frac d2}(1-w^2)^{\gamma-1}J_{\frac d2-1}(bw)\de w=2^{\gamma-1}\Gamma(\gamma)b^{-\gamma}J_{\gamma+\frac d2-1}(b),
\end{equation}
with $b=ct\rho^{\nu/2}, \nu=\frac d2-1$ and $\mu=\gamma-1$. By inserting  \eqref{eq:intbess} into \eqref{eq:fcf1}, we get
\begin{align*}
u_\nu(\xd,t)&=\frac{1}{(2\pi)^{d/2}}\left(\frac{2}{ct}\right)^{\gamma+\frac d2-1}\frac{\Gamma\left(\gamma+\frac d2\right)}{||\xd||^{\frac d2-1}}\int_0^\infty\rho^{\frac d2-\frac\nu2(\frac d2-1+\gamma)}J_{\gamma+\frac d2-1}\left(ct\rho^{\nu/2}\right)J_{\frac d2-1}(\rho||\xd||)\de\rho\\
&=\frac{1}{\pi^{d/2}(ct)^{\frac d2-1}}\frac{\Gamma\left(\gamma+\frac d2\right)}{\Gamma(\gamma)||\xd||^{\frac d2-1}}\\
&\quad\times\int_0^\infty\rho^{\frac d2+\frac \nu2-\frac{\nu d}{4}}J_{\frac d2-1}(\rho||\xd||)\de \rho\int_0^1 w^{\frac d2}(1-w^2)^{\gamma-1}J_{\frac d2-1}\left(ct\rho^{\nu/2}w\right)\de w
\end{align*}
By a simple change of variable in the above integral we obtain the result \eqref{eq:fcf2}.
\end{proof}

An alternative form of the characteristic function \eqref{eq:ffsolepd} of the solution to the fractional EPD equation is given in the next theorem and this inspires an alternative representation of the solution $u_\nu$ in terms of one-dimensional random flights arising from \eqref{eq:epd}.

\begin{theorem}\label{teo:cff}
The characteristic function \eqref{eq:ffsolepd} can be written as
\begin{align}\label{eq:ffsolepd2}
\hat u_\nu(\xi_d,t)
&=\frac{\Gamma(\gamma+\frac d2)}{\sqrt{\pi}\Gamma(\frac d2+\gamma-\frac12)ct}\int_{-ct}^{ct}\left(1-\frac{w^2}{c^2t^2}\right)^{\frac d2+\gamma-\frac12-1}\left(\frac{e^{i ||\xi_d||^{\nu/2}w}+e^{-i ||\xi_d||^{\nu/2}w}}{2}\right)\de w.
\end{align}
\end{theorem}
\begin{proof}
The Poisson integral representation of the Bessel functions reads
\begin{equation}\label{eq:poisson}
J_\mu(z)=\frac{(z/2)^\mu}{\sqrt\pi \Gamma(\mu+\frac12)}\int_{-1}^{+1}(1-w^2)^{\mu-\frac12}\cos(z w)\de w
\end{equation}
valid for $\mu>-\frac12,z\in\mathbb R$ (see \cite{lebedev}, pag. 114, formula (5.10.3)).
By inserting \eqref{eq:poisson} into \eqref{eq:ffsolepd}, we readily have that
\begin{align*}
\hat u_\nu(\xi_d,t)&=\frac{\Gamma(\gamma+\frac d2)}{\sqrt{\pi}\Gamma(\frac d2+\gamma-\frac12)}\int_{-1}^1(1-w^2)^{\frac d2+\gamma-\frac12-1}\cos(ct ||\xi_d||^{\nu/2}w)\de w\\
&=\frac{\Gamma(\gamma+\frac d2)}{\sqrt{\pi}\Gamma(\frac d2+\gamma-\frac12)}\int_{-1}^1(1-w^2)^{\frac d2+\gamma-\frac12-1}\left(\frac{e^{ict ||\xi_d||^{\nu/2}w}+e^{-ict ||\xi_d||^{\nu/2}w}}{2}\right)\de w\\
&=\frac{\Gamma(\gamma+\frac d2)}{\sqrt{\pi}\Gamma(\frac d2+\gamma-\frac12)ct}\int_{-ct}^{ct}\left(1-\frac{w^2}{c^2t^2}\right)^{\frac d2+\gamma-\frac12-1}\left(\frac{e^{i ||\xi_d||^{\nu/2}w}+e^{-i ||\xi_d||^{\nu/2}w}}{2}\right)\de w.
\end{align*}

\end{proof}

\begin{remark}

Since the stable subordinator process has characteristic function \eqref{eq:chstabsub}, the function $e^{\pm i ||\xi_d||^{\nu/2}w}$ seems to be related to the Fourier transform of ${\bf S}_\nu(t).$ We observe that
\begin{align}\label{eq:cfcauchy2}
\frac{P({\bf S}_\nu(t)\in \de\xd)}{\de\xd}&=\frac{1}{(2\pi)^d}\int_{\mathbb R^d}e^{-i\langle \xi_d,\xd \rangle}e^{-t||\xi_d||^{\nu}}\de\xi_d\\
&=\frac{1}{(2\pi)^{d/2}}\int_0^\infty \rho^{d-1}\frac{J_{\frac d2-1}(\rho||\xd||)}{(\rho||\xd||)^{\frac d2-1}}e^{-t\rho^{\nu}}\de \rho\notag\\
&=\frac{1}{(2\pi)^{d/2}}\frac{1}{||\xd||^{\frac d2-1}}\sum_{k=0}^\infty (-1)^k\left(\frac{||\xd||}{2}\right)^{2k+\frac d2-1}\frac{1}{k! \Gamma(k+\frac d2)}\int_0^\infty \rho^{d+2k-1}e^{-t\rho^{\nu}}\de \rho\notag\\\
&=\frac{1}{(2\pi)^{d/2}}\sum_{k=0}^\infty (-1)^k\left(\frac{1}{2}\right)^{2k+\frac d2-1}||\xd||^{2k}\frac{\Gamma(d+2k)}{k! \Gamma(k+\frac d2)t^{\frac{2k+d}{\nu}}}.\notag
\end{align}

For $\nu=1,$ by inserting the following formula
$$\binom{-\frac{d+1}{2}}{k}=\frac{(-1)^k\Gamma(d+2k)\Gamma(\frac d2+1)}{\Gamma(\frac d2+k)2^{2k-1}k!\Gamma(d+1)}$$
into \eqref{eq:cfcauchy2}, we obtain
\begin{align}\label{eq:cauchy}
\frac{P({\bf S}_d(t)\in \de\xd)}{\de\xd}&=\frac{1}{(2\pi)^{d/2}}\frac{1}{2^{\frac d2}t^d}\sum_{k=0}^\infty\binom{-\frac{d+1}{2}}{k} \left(\frac{||\xd||^2}{t^2}\right)^k\frac{\Gamma(d+1)}{\Gamma(\frac d2+1)}\\\notag
&=\frac{1}{(2\pi)^{d/2}} \frac{1}{2^{\frac d2}t^d}\frac{\Gamma(d+1)}{\Gamma(\frac d2+1)}\frac{1}{\left(1+\frac{||\xd||^2}{t^2}\right)^{\frac{d+1}{2}}}
\end{align}
Thus, for $\nu=1,$ \eqref{eq:chstabsub} becomes the characteristic function of a $d$-dimensional Cauchy process $({\bf C}_d(t), t\geq 0)$ having probability distribution equal to \eqref{eq:cauchy}.
Therefore the density function \eqref{eq:cauchy} can be interpreted as the distribution of the hitting point of a $(d+1)$-dimensional Brownian motion on the subspace $\mathbb S_1^{d}$ with starting point $(0,...,0,t)\in\mathbb R^{d+1}.$ For the sake of completeness we remind that \eqref{eq:cauchy} solves the Laplace equation
$$\frac{\partial^2p}{\partial t^2}+\Delta p=0.$$ 
\end{remark}

\begin{remark}

By Bochner's subordination it is possible to show that
\begin{equation}
{\bf S}_\nu(t)\stackrel{(\text{law})}{=}{\bf B}_d(Y_{\nu/2}(t)), \quad 0<\nu\leq 2,
\end{equation}
where $({\bf B}_d(t), t\geq 0)$ is a standard $d$-dimensional Brownian motion and $(Y_{\nu/2}(t),t\geq 0)$ is an independent $\nu/2$-stable subordinator. Equivalently, we have that
\begin{equation}
{\bf S}_\nu(t)\stackrel{(\text{law})}{=}{\bf C}_d(Y_{\nu}(t)), \quad 0<\nu< 1,
\end{equation}
where $({\bf C}_d(t),t\geq 0)$ and  $(Y_{\nu}(t),t\geq 0)$ are independent.
Indeed
\begin{align*}
E\left[e^{i\langle \xi_d,{\bf C}_d(Y_{\nu}(t))\rangle}\right]&=E\left[E\left[e^{i\langle \xi_d,{\bf C}_d(Y_{\nu}(t))\rangle}| \mathcal{F}_{Y_{\nu}}\right]\right]\\
&=E\left[ e^{-Y_{\nu}(t)||\xi_d||}\right]=e^{-t ||\xi_d||^\nu},
\end{align*}
where $\mathcal{F}_{Y_{\nu}}$ is the natural filtration of $({\bf C}_d(t),t\geq 0)$ stopped at $Y_{\nu}(t).$
\end{remark}

Our task now is to obtain the inverse Fourier transform of \eqref{eq:ffsolepd2}.
The term
$e^{\pm i||\xi_d||^{\nu/2}w}\hat\varphi(\xi_d)$ 
is the Fourier transform of the solution to the following the Cauchy problem involving the fractional Schr\"odinger equation
\begin{equation}\label{eq:schrod}
\begin{cases}
 \pm i\frac{\partial u}{\partial w}=-(-\Delta)^{\nu/4}u,\\
   u(\xd,0)=\varphi(\xd) ,
   \end{cases}
\end{equation} 
where $\varphi(\xd)\in\mathcal S(\re^d).$
Therefore, we can say that  $\frac{e^{i||\xi_d||^{\nu/2}w}+e^{-i||\xi_d||^{\nu/2}w}}{2}\hat\varphi(\xi_d)$ is the Fourier transform of the solution to
\begin{align}\label{eq:cauchypr}
\begin{cases}
\left(\frac{\partial^2}{\partial w^2}+(-\Delta)^{\nu/2}\right)u=\left (i\frac{\partial }{\partial w}+(-\Delta)^{\nu/4}\right)\left(-i\frac{\partial }{\partial w}+(-\Delta)^{\nu/4}\right)u=0,\\
u(\xd,0)=\varphi(\xd),\\
\left. \frac{\partial u(\xd,w)}{\partial w}\right|_{w=0}=0.
 \end{cases}
\end{align}

Let us denote by $p_1^\nu:=p_1^\nu(\xd,w)$ and $p_2^\nu:=p_2^\nu(\xd,w)$ the solutions to the Cauchy problems \eqref{eq:schrod}. 
We have that
\begin{align}\label{eq:op1}
p_1^\nu(\xd,w)&=\frac{1}{(2\pi)^d}\int_{\mathbb R^d} e^{-i\langle \xi_d,\xd\rangle}e^{i ||\xi_d||^{\nu/2}w}\hat\varphi(\xi_d)\de\xi_d
\end{align}
Analogously, we get
\begin{align}\label{eq:op2}
p_2^\nu(\xd,w)&=\frac{1}{(2\pi)^d}\int_{\mathbb R^d} e^{-i\langle \xi_d,\xd\rangle}e^{-i ||\xi_d||^{\nu/2}w}\hat\varphi(\xi_d)\de\xi_d.
\end{align}

We obtain the following result relating the EPD equation with the random flights.
\begin{corollary}
The solution of the Cauchy problem \eqref{fracepd} is given by
\begin{equation}\label{eq:solfepd}
u_\nu(\xd,t)=\int_{-ct}^{ct}g(w,t) \left(\frac{p_1^\nu(\xd,w)+p_2^\nu(\xd,w)}{2}\right)\de w,
\end{equation}
where
$$g(w,t)=\frac{\Gamma(\gamma+\frac d2)}{\sqrt{\pi}\Gamma(\frac d2+\gamma-\frac12)ct}\left(1-\frac{w^2}{c^2t^2}\right)_+^{\frac d2+\gamma-\frac12-1}$$
is the solution to the one-dimensional EPD equation \eqref{eq:epd} with $\lambda(t)=\frac{\gamma+\frac d2-\frac12}{2t}$.
\end{corollary}
\begin{proof}
The result \eqref{eq:solfepd} follows from \eqref{eq:ffsolepd2},  \eqref{eq:op1} and \eqref{eq:op2} when we consider $\varphi({\bf x}_d)$ as initial condition in the Cauchy problem \eqref{fracepd}. Furthermore, since $\varphi({\bf x}_d)\in\mathcal S(\mathbb R^d),$ by Plancharel's theorem we can conclude that $p_1^\nu(\xd,w),p_1^\nu(\xd,w)\in L^2(\mathbb R^d)$ and then $u_{\nu}({\bf x}_d,t)\in L^2(\mathbb R^d).$
\end{proof}

\begin{remark}
The result \eqref{eq:solfepd} shows that the solution of the Cauchy problem of the fractional EPD equation represents the Erd\'elyi-Kober integral of the solution of the fractional wave equation \eqref{eq:cauchypr}. We recall that the Erd\'elyi-Kober integral is defined as
$$(I_\alpha^{m}f)(x)=\frac{m}{\Gamma(\alpha)}\int_0^x (x^m-y^m)^{\alpha-1}y^{m-1}f(y)\de y,\quad \alpha>0,m>0.$$
\end{remark}

\begin{remark}
Constructing probabilistic solutions to the Schr\"odinger equation has been undertaken in \cite{isf}. In this paper the authors study a probabilistic interpretation of the solution to the Cauchy problem
\begin{align}\label{cauchysc}
&\frac{\partial u}{\partial t}=\frac{\sigma^2}{2}\Delta u,\quad u(\xd,0)=\varphi(\xd),\quad \xd\in \mathbb R^d,t>0,
\end{align}
where $\sigma$ is a complex number with Re $\sigma^2\geq 0.$ For $\sigma^2=i$ equation \eqref{cauchysc} becomes the Schr\"odinger equation. In Theorems 1,2 and 3, the authors show that the solution to \eqref{cauchysc} can be written as the mean value of a functional of the Brownian motion $({\bf B}_d(t),t\geq 0)$; that is
\begin{equation}
u(\xd,t)=E\varphi(\xd+\sigma {\bf B}_d(t)),
\end{equation}
if $\varphi$ belongs to a suitable class of functions on the $d$-dimensional complex space $\mathbb C^d$ (depending on Re $\sigma^2>0$ and Re $\sigma^2=0$, see \cite{isf}). 

In the same spirit of \cite{isf}, we suggest a probabilistic interpretation to the solution of the following fractional Cauchy problem
\begin{align}\label{fcauchysc}
&\frac{\partial u_\nu}{\partial t}=-\sigma^2(-\Delta)^{\nu/4} u_\nu,\quad u_\nu(\xd,0)=\varphi(\xd),\quad \xd\in \mathbb R^d,t>0,
\end{align}
where $\sigma\in\mathbb C$ with Re $\sigma^2\geq 0$ and $0<\nu<2.$ If we write the solution to \eqref{fcauchysc} as
\begin{equation}
u_\nu(\xd,t)=E\varphi(\xd+\sigma {\bf S}_{\nu/2}(t)),
\end{equation}
where ${\bf S}_{\nu/2}(t)\stackrel{(\text{law})}{=}{\bf B}_d(Y_{\nu/4}(t)),$ for $0<\nu<2,$ or ${\bf S}_{\nu/2}(t)\stackrel{(\text{law})}{=}{\bf C}_d(Y_{\nu/2}(t)),$ for $0<\nu<1,$ we obtain a probabilistic  representation of the solution of \eqref{fcauchysc} by considering a suitable class of functions $\varphi$ on the complex space $\mathbb C^d.$ For $\sigma^2=\mp i,$ \eqref{fcauchysc} reduces to \eqref{eq:schrod}.

We note that there is a huge recent literature on the time and space fractional Schr\"odinger equation (including the non-linear case). See, for example, the recent papers \cite{fschrod} and \cite{su}. Furthermore for some connections between the generalized Schr\"odinger equation and L\'evy processes see \cite{cuf}. 
\end{remark}

\section{A relationship between the fractional EPD equation and the PME}

The solution to the fractional EPD equation \eqref{fracepd}, after a suitable time-change can be related to some form of fractionalized PME as we now show. By changing the time scale as $t=t'^{\beta}$ (and for the reader's convenience we indicate the new time coordinate $t'$ by $t$) and by setting $c=1/\sqrt B,$ we are able to write \eqref{eq:solfepd} as follows
\begin{equation}\label{eq:fracpme}
\mathfrak u_\nu(\xd,t):= u_\nu(\xd,t^{\beta})=\int_{-\frac{t^{\beta}}{\sqrt B}}^{\frac{t^{\beta}}{\sqrt B}}\mathfrak g(w,t) \left(\frac{p_1^\nu(\xd,w)+p_2^\nu(\xd,w)}{2}\right)\de w
\end{equation}
where 
\begin{equation}
\mathfrak g(w,t):=g(w,t^{\beta})=\frac{\Gamma(\gamma+\frac d2)B}{\sqrt{\pi}\Gamma(\frac d2+\gamma-\frac12)t^\beta}\left(1-\frac{ Bw^2}{t^{2\beta}}\right)_+^{\frac d2+\gamma-\frac12-1}
\end{equation}
is the Barenblatt solution to the one-dimensional PME with $m=1+\frac{2}{d+2\gamma-3}, \beta=\frac{1}{1+m}$ and $B=\frac{m-1}{2m(m+1)}.$ The parameter $m$ is strictly greater than 1 for all $d\geq 3.$ For $d=1$ and $d=2,$ $m>1$ if $\gamma>1$ and $\gamma>\frac12,$ respectively. As shown in Theorem \ref{teo:sppme}, $\mathfrak g(w,t)$ is related to some type of random flights. Therefore, the next result allows to connect the fractional EPD equation with the PME by means of the function $\mathfrak g(w,t).$

We have now the following theorem.

\begin{theorem}
The function $\mathfrak u_\nu(\xd,t)$ solves the following fractional equation

\begin{align}\label{eq:eqfracpme}
\frac{\partial\mathfrak u_\nu (\xd,t)}{\partial t}
=-(-\Delta)^{\nu/2}\int_{-\frac{t^{\beta}}{\sqrt B}}^{\frac{t^{\beta}}{\sqrt B}} (\mathfrak g(w,t) )^m\left( \frac{  p_1^\nu(\xd,w)+  p_2^\nu(\xd,w)}{2}\right)\de w,
\end{align}
with $0<\nu\leq 2.$
\end{theorem}
\begin{proof}
Since $\mathfrak g(\pm \frac{t^{\beta}}{\sqrt B},t)=0,$ from \eqref{eq:fracpme} we have that
$$\frac{\partial\mathfrak u_\nu (\xd,t)}{\partial t}=\int_{-\frac{t^{\beta}}{\sqrt B}}^{\frac{t^{\beta}}{\sqrt B}}\frac{\partial \mathfrak g(w,t)}{\partial t} \left(\frac{p_1^\nu(\xd,w)+p_2^\nu(\xd,w)}{2}\right)\de w.$$

Therefore, bearing in mind that $\mathfrak g$ is the fundamental solution of the PME \eqref{eq:pme} for $d=1,$ the following equality holds
\begin{align}\label{eq:fracint}
\frac{\partial\mathfrak u_\nu (\xd,t)}{\partial t}&=\int_{-\frac{t^{\beta}}{\sqrt B}}^{\frac{t^{\beta}}{\sqrt B}}\frac{\partial^2 (\mathfrak g(w,t) )^m}{\partial w^2} \left(\frac{p_1^\nu(\xd,w)+p_2^\nu(\xd,w)}{2}\right)\de w\\
&=-\frac 12\int_{-\frac{t^{\beta}}{\sqrt B}}^{\frac{t^{\beta}}{\sqrt B}}\frac{\partial   (\mathfrak g(w,t) )^m}{\partial w} \left(\frac{\partial p_1^\nu(\xd,w)}{\partial w}+\frac{\partial p_2^\nu(\xd,w)}{\partial w}\right)\de w\notag,
\end{align}
where in the last step we have used the fact that $\frac{\partial \mathfrak g^m(\pm \frac{t^{\beta}}{\sqrt B},t^{\beta})}{\partial w}=0$ and the quantities $p_1^\nu(\xd,\pm \frac{t^{\beta}}{\sqrt B})$ and  $p_2^\nu(\xd,\pm \frac{t^{\beta}}{\sqrt B})$ are bounded. Since $p_1^\nu$ and $p_2^\nu$ satisfy the equation \eqref{eq:op1} and \eqref{eq:op2}, respectively, we can write \eqref{eq:fracint} as
\begin{align}\label{eq:fracint2}
\frac{\partial\mathfrak u_\nu (\xd,t)}{\partial t}&=-\frac{1}{2}\int_{-\frac{t^{\beta}}{\sqrt B}}^{\frac{t^{\beta}}{\sqrt B}}\frac{\partial (\mathfrak g(w,t) )^m}{\partial w} \left(i(-\Delta)^{\nu/4} p_1^\nu(\xd,w)-i (-\Delta)^{\nu/4}p_2^\nu(\xd,w)\right)\de w\\
&=-\frac{i(-\Delta)^{\nu/4}}{2}\int_{-\frac{t^{\beta}}{\sqrt B}}^{\frac{t^{\beta}}{\sqrt B}}\frac{\partial   (\mathfrak g(w,t) )^m}{\partial w} \left( p_1^\nu(\xd,w)- p_2^\nu(\xd,w)\right)\de w,\notag
\end{align}
where the last step follows from the following observation: since $p_1^\nu(\xd,w),p_1^\nu(\xd,w)\in L^2(\mathbb R^d),$ then the integral appearing in the second line of \eqref{eq:fracint2} belongs to $L^2(\mathbb R^d).$
A further integration by parts in \eqref{eq:fracint2} yields
\begin{align*}
\frac{\partial\mathfrak u_\nu (\xd,t)}{\partial t}
&=\frac{i(-\Delta)^{\nu/4}}{2}\int_{-\frac{t^{\beta}}{\sqrt B}}^{\frac{t^{\beta}}{\sqrt B}} \mathfrak  (\mathfrak g(w,t) )^m \left( \frac{ \partial p_1^\nu(\xd,w)}{\partial w}- \frac{ \partial p_2^\nu(\xd,w)}{\partial w}\right)\de w\\
&=-(-\Delta)^{\nu/4}(-\Delta)^{\nu/4}\int_{-\frac{t^{\beta}}{\sqrt B}}^{\frac{t^{\beta}}{\sqrt B}}  (\mathfrak g(w,t) )^m \left( \frac{  p_1^\nu(\xd,w)+  p_2^\nu(\xd,w)}{2}\right)\de w,
\end{align*}
where the last step is justified analogously to \eqref{eq:fracint2}.
Finally, by applying the semigroup property for the fractional Laplace operator, we obtain the result \eqref{eq:eqfracpme}.
\end{proof}

\begin{remark}
In the special case where $\nu=2$ and $\varphi(\xd)=\delta(\xd),$ we obtain the PME. Indeed, in this case $p_1^2(\xd,w)=\delta(||\xd||+w)$ and $p_2^2(\xd,w)=\delta(||\xd||-w)$ and then
\begin{align*}
\frac{\partial\mathfrak u_2 (\xd,t)}{\partial t}
&=\Delta\int_{-\frac{t^{\beta}}{\sqrt B}}^{\frac{t^{\beta}}{\sqrt B}}  (\mathfrak g(w,t) )^m \left( \frac{ \delta(||\xd||-w)+  \delta(||\xd||+w)}{2}\right)\de w\\
&=\Delta\left( \frac{  (\mathfrak g(||\xd||,t))^m+ ( \mathfrak g(-||\xd||,t))^m}{2}\right)	\\
&=\Delta (\mathfrak u_2 (\xd,t))^m.
\end{align*}
\end{remark}

\section{Solutions to the higher-order EPD equations}

We now turn back to the solution \eqref{fracepd} of the EPD equation and consider the case where $\nu>2, d=1.$ Our starting point here is the fractional equation
\begin{equation}
\left(\frac{\partial}{\partial t}-i\frac{\partial^\nu}{\partial |x|^\nu}\right)\left(\frac{\partial}{\partial t}+i\frac{\partial^\nu}{\partial |x|^\nu}\right)u=0
\end{equation}
which for integer values of $\nu=n,n\in \mathbb N$ and $c_n>0,$ reduces to the form
\begin{equation}\label{eq:higpde}
\frac{\partial^2 u}{\partial t^2}+c_n\frac{\partial^{2n}u}{\partial x^{2n}}=0.
\end{equation}
Equation \eqref{eq:higpde} for $n=2, c_n=1$ becomes the famous equation of vibrations of rods.

The solutions of 
\begin{align}\label{eq:hosyst}
\begin{cases}
\left(\frac{\partial}{\partial t}-c_n^{1/2}\frac{\partial^n}{\partial x^n}\right)u=0, \\
\left(\frac{\partial}{\partial t}+c_n^{1/2}\frac{\partial^n}{\partial x^n}\right) u=0,
\end{cases}
\end{align}
are solutions of
\begin{equation}\label{eq:hoe}
\left(\frac{\partial}{\partial t}-c_n^{1/2}\frac{\partial^n}{\partial x^n}\right)\left(\frac{\partial}{\partial t}+c_n^{1/2}\frac{\partial^n}{\partial x^n}\right)u=0,
\end{equation}
which, for special values of $c_n$ yield higher-order Schr\"odinger equations. Clearly the Fourier transforms of \eqref{eq:hosyst} become
\begin{align*}
\begin{cases}
\frac{\partial \hat u}{\partial t}=c_n^{1/2}(-i\xi)^n\hat u \\
\frac{\partial \hat u}{\partial t}=-c_n^{1/2}(-i\xi)^n\hat u
\end{cases}
\end{align*}
so that
\begin{align*}
\hat u(\xi,t)=\frac12\left(e^{c_n^{1/2}(-i\xi)^n t}+e^{-c_n^{1/2}(-i\xi)^n t}\right)=\cos(c_n^{1/2}e^{-in\frac\pi 2}|\xi|^n t)
\end{align*}
solves the equation emerging from the Fourier transform of \eqref{eq:hoe} with initial condition $u(x,0)=\delta(x).$ The solution $u$ takes the form
\begin{align}\label{eq:solho}
u(x,t)&=\frac{1}{2\pi}\int_{-\infty}^{+\infty} e^{-i\xi x}\cos(c_n^{1/2}e^{-in\frac\pi 2}|\xi|^n t)\de\xi\notag\\
&=\frac{1}{\pi}\int_{0}^{+\infty} \cos(\xi x)\cos(c_n^{1/2}e^{-in\frac\pi 2}|\xi|^n t)\de\xi.
\end{align}
With the choice of $c_n^{1/2}=e^{i\pi n/2}=(-1)^{n/2},$ the solution \eqref{eq:solho} reduces to
\begin{equation}\label{eq:intpseud}
u(x,t)=\frac{1}{\pi}\int_{0}^{+\infty} \cos(\xi x)\cos(|\xi|^n t)\de\xi.
\end{equation}

For $n=2,$ the integral \eqref{eq:intpseud} yields
\begin{equation}\label{eq:intpseud2}
u(x,t)=\frac{1}{\pi}\int_{0}^{+\infty} \cos(\xi x)\cos(|\xi|^2 t)\de\xi=\frac{1}{\sqrt{2\pi t}}\cos\left(\frac{x^2}{2t}-\frac\pi 4\right),
\end{equation}
and represents the fundamental solution to the equation of vibrations of rods. 

Pseudoprocesses related to higher-order heat equations $\frac{\partial u}{\partial t}=c_n\frac{\partial^{n}u}{\partial x^{n}}$ have been introduced in several papers mimicking the construction of the Wiener measure. Thus,  
a signed measure $P$ is constructed on a set of real-valued functions $t : t\in [0,\infty)\to x(t)$ called the sample paths of the process, in a manner similar to that described, for instance, by \cite{kry} and \cite{dal}. In particular, for cylinder sets $C$ of the form
$$C:=\{x: a_j\leq x(t)\leq b_j, j=1,2,...,n\},\quad 0<t_1<..<t_j<...<t_n,$$
the measure $P$ is defined by
\begin{align}\label{eq:measupse}
P(C)=\int_{a_1}^{b_1}...\int_{a_n}^{b_n}\prod_{j=1}^n u(x_j-x_{j-1},t_j-t_{j-1})\de x_j
\end{align}
where $t_0=0, x_0=0$ and $x_j:=x(t_j).$ The measure $P$ is countably additive on the field of sets generated by $x(t_j), j = 1, 2, . , n,$ for fixed $t_j$ and $n$ finite. 
Pseudoprocesses related to \eqref{eq:higpde} for $n=2$ and $c_n=-1/4$ have been studied in \cite{od} and the corresponding measure on cylinder sets becomes

$$P(C)=\int_{a_1}^{b_1}...\int_{a_n}^{b_n} \frac{(2\pi)^{-n/2}}{\prod_{j=1}^n\sqrt{t_j-t_j}}\cos\left(\sum_{j=1}^n \frac{(x_j-x_{j-1})^2}{2(t_j-t_{j-1})}-n\frac\pi4\right)\de x_j.$$

For $n=3, c_n=\mp 1,$ from \eqref{eq:intpseud} we have that
\begin{align*}
u(x,t)&=\frac{1}{\pi}\int_{0}^{+\infty} \cos(\xi x)\cos(|\xi|^3 t)\de\xi\\
&=\frac{1}{\pi}\int_{0}^{+\infty} [\cos (\xi x+\xi^3t)+\cos (\xi x-\xi^3t)]\de\xi\\
&=\frac{1}{(3t)^{1/3}}\left[Ai\left(\frac{x}{(3t)^{1/3}}\right)+Ai\left(-\frac{x}{(3t)^{1/3}}\right)\right]
\end{align*}
where $Ai(x):=\frac1\pi\int_0^\infty\cos \left(\frac{t^3}{3}+xt\right)\de t$ is the Airy function of first order.

In light of Theorem \ref{teo:cff}, the solution to the higher-order EPD equation
\begin{equation}
\frac{\partial^2 u}{\partial t^2}+\frac{2\lambda}{t}\frac{\partial u}{\partial t}=c_n\frac{\partial^{2n}u}{\partial x^{2n}},\quad \lambda >0,
\end{equation}
is the law of the composition $X_{2n}(T(t)),$ where $(X_{2n}(t),t\geq 0)$ is the pseudoprocess whose signed measure is 
\begin{equation}\label{eq: pseud}
u(x,t)=\frac{1}{\pi}\int_{0}^{+\infty} \cos(\xi x)\cos(|\xi|^{2n} t)\de\xi,
\end{equation}
and $(T(t),t\geq 0)$ is the telegraph process having probability density function given by the solution $g(x,t)$ to the one-dimensional EPD equation. For a more detailed description of pseudoprocesses see, for example, \cite{lachal}. We were able to obtain explicitly \eqref{eq: pseud} in the special cases $n=2,3$ as shown above.

\section*{Appendix}

In the derivation of the Kompanets-Zel'dovich-Barenblatt solution \eqref{eq:bsolbis}, one must consider a function
\begin{equation}\label{eq:bsol}
f(\xd,t)=t^\delta\left(1-B\frac{||\xd||^2}{t^\eta}\right)_+^\gamma,
\end{equation}
where $\delta,\eta,\gamma,B$ represent real constant,
and compare the following derivatives for $f:=f(\xd,t)$
\begin{align*}
&\frac{\partial f}{\partial t}=t^{\delta -1}(\delta -\eta \gamma)\left(1-B\frac{||x||^2}{t^\eta}\right)_+^\gamma+\eta t^{\delta -1}\left(1-B\frac{||x||^2}{t^\eta}\right)_+^{\gamma-1},\end{align*} 
\begin{align*}
\Delta (f^m)&=-2Bt^{m\delta -\eta}\gamma m\left(1-B\frac{||x||^2}{t^\eta}\right)_+^{\gamma m-1}(d+2(\gamma m-1))\notag\\
&\quad+\gamma m(\gamma m-1)4B t^{m\delta -\eta}\left(1-B\frac{||x||^2}{t^\eta}\right)_+^{\gamma m-2}
\end{align*} 
In order to prove that \eqref{eq:bsol} satisfies \eqref{eq:pme} we have to impose that
\begin{equation}
\begin{cases}
\gamma=\gamma m-1\\
\delta-1=m\delta-\eta\\
(\delta-\eta\gamma)=-\gamma m2 B(d+2(\gamma m-1))\\
\gamma\eta =\gamma m(\gamma m-1)4B
\end{cases}
\end{equation}
from which we derive
$$\gamma=\frac{1}{m-1},\quad \delta=-\frac{d}{2+d(m-1)},\quad \eta=\frac{2}{2+d(m-1)}, \quad B:=\frac{\alpha(m-1)}{2md}=\frac{m-1}{2m(2+d(m-1))}.$$

\end{document}